\documentclass[12pt, reqno]{amsart}
\usepackage{amsfonts,mathrsfs,amssymb,amsmath,amsthm,bm,url,enumerate}
\newtheorem{thma}{Lemma}
\newtheorem{thmb}[thma]{Theorem}

\makeatletter
\def\imod#1{\allowbreak\mkern10mu({\operator@font mod}\,\,#1)}
\makeatother
\newcommand{\txt}{\textrm}
\newcommand{\xx}{\textbf{x}}

\newcommand{\kkk}{\textbf{k}}

\newcommand{\yy}{\textbf{y}}

\newcommand{\beq}{\begin{equation}}
\newcommand{\eeq}{\end{equation}}
\newcommand{\beqq}{\begin{equation*}}
\newcommand{\eeqq}{\end{equation*}}
\newcommand{\bal}{\begin{align}}
\newcommand{\eali}{\end{align}}
\newcommand{\ball}{\begin{align*}}
\newcommand{\ealii}{\end{align*}}
\newcommand{\eps}{\varepsilon}

\numberwithin{equation}{section}
\numberwithin{equation}{section}
\begin{document}
\title{A Note On a Theorem of Heath-Brown and Skorobogatov}
\author{Mike Swarbrick Jones}
\address{  School of Mathematics, Bristol University, Bristol, BS8  1TW}  

\email{mike.swarbrickjones@bristol.ac.uk}

\begin{abstract} We generalise a result of Heath-Brown and Skorobogatov \cite{HBS} to show that a certain class of varieties over a number field $k$ satisfies Weak Approximation and the Hasse Principle, provided there is no Brauer-Manin obstruction.  

\end{abstract}

\maketitle

\section{Introduction}

Let $k$ be a number field with $[k: \mathbb{Q}]=m$, and ring of integers $\mathfrak{o}$.  Let $K$ be a finite extension of $k$ with $[K: k]=n$, and let $\tau_1, \dots , \tau_n$ be a $k$-basis of $K$.  For $\xx \in k^n$, we let $N(\xx) = N_{K/k}( x^{(1)} \tau_1 + \cdots + x^{(n)} \tau_n)$ be a norm form of $K/k$.  
The subject of this note is the affine variety $X$, defined by the Diophantine equation $$P(t)= N(\xx),$$  where $P(t)$ is a polynomial with coefficients in $k$.
Let $\overline{k}$ be an algebraic closure of $k$.  If $P(t)$ has exactly two solutions in $k$, and no other roots in $\overline{k}$, then we can immediately change variables to obtain the equation \beq\label{maineq} t^{a_0}(1-t)^{a_1}= \alpha N( \xx),\eeq where $\alpha \in k^{*}$ and $a_0, a_1$ are positive integers.  The culmination of \cite{CHS} and \cite{HBS} is the following theorem, under the additional assumption that $k= \mathbb{Q}$:
 
\begin{thmb}\label{mainthm}The Brauer-Manin obstruction is the only obstruction to the Hasse Principle and Weak Approximation on any smooth projective model of the open subset of the variety $\eqref{maineq}$, given by  $P(t) \neq 0$. \end{thmb}  

There was only a modest link missing to show this theorem for general $k$, which is straightforward by present standards, and is our aim here.  The key step of \cite{CHS} and \cite{HBS} is a descent argument, which reduces the problem to showing the validity of the Hasse principle and weak approximation on the smooth affine quasi-projective variety $Y \subset \mathbb{P}^{2n}$ defined by 

\beq\label{equation1}a N(\xx) + b N(\yy) = z^n \neq 0 , \eeq for given $a,b \in \mathfrak{o}$.  In \cite{HBS} this was achieved by finding an asymptotic lower bound for the number of suitably constrained integer solutions to $\eqref{equation1}$ in a large box.  The principle tool was the Hardy-Littlewood circle method for $k=\mathbb{Q}$.  We shall use a more general version of the circle method here to handle arbitrary number fields.

In \cite{CHS}, the Brauer group of the variety $X$ was calculated for some special cases, to identify some situations where the Brauer--Manin obstruction is empty.  For example if $a_0$ and $a_1$ are coprime, and $K/k$ does not contain any non trivial cyclic extension of $k$, then $\txt{Br}(k)=\txt{Br}(X)$, and so the Hasse principle and weak approximation both hold.  On the other hand, it is known that there can be obstructions to weak approximation if $K$ is a cyclic extension of $k$.  For an example due to Coray, see \cite[\S 9]{CTSal}.

\textbf{Acknowledgements:}  \linebreak
I'd like to thank my supervisor Tim Browning for improving the quality of the paper substantially.  Also thanks to Tony V\'arilly-Alvarado for showing interest in the problem.

\section{Notation}Let $\mathfrak{o}$ be the ring of integers of $k$.  Without loss of generality, suppose that $\tau_1, \dots , \tau_n$ is a $\mathfrak{o}$-basis of $K$.  Let $\mathfrak{n}$ be an integral ideal of $\mathfrak{o}$, with $\mathbb{Z}$-basis $\omega_1, \dots , \omega_m$.  Let $\sigma_1, \dots \sigma_{n_1}$ be the distinct real embeddings of $k$, and let $\sigma_{n_1+1}, \dots \sigma_{n_1 + 2n_2}$ be the distinct complex embeddings, such that $\sigma_{n_1+i}$ is conjugate to $\sigma_{n_1+n_2+i}$.  Put $k_i$ to be the completion of $k$ with respect to the embedding $\sigma_i$, for $i=1, \dots , n_1 + n_2$.

Define $V$ to be the commutative $\mathbb{R}$-algebra $\oplus^{n_1+n_2}_{i=1}k_{i}  \cong k \otimes_\mathbb{Q} \mathbb{R}$.  
For an element $x \in V$, we write $\pi_i(x)$ for its projection onto the $i$th summand, (so $x= \oplus \pi_i(x)$).  There is a canonical embedding of $k$ into $V$ given by $\alpha \rightarrow \oplus\sigma_i(\alpha)$.  We identify $k$ with its image in $V$.  Under this image, $\mathfrak{n}$ forms a lattice in $V$, and $\omega_1, \dots , \omega_n$ form a real basis for $V$.  We define trace and norm maps on $V$ as 

\begin{equation*}
\textrm{Tr} ({\alpha}) = \sum_{i=1}^{n_1} \pi_i(\alpha) +  2 \sum_{i=n_1+1}^{n_1+n_2} \mathfrak{R}( \pi_i(\alpha)),
\end{equation*}

\begin{equation*}
\txt{Nm} ({\alpha})= \prod_{i=1}^{n_1} \pi_i(\alpha) \prod_{i=n_1+1}^{n_1+n_2} |\pi_i(\alpha)|^2 ,
\end{equation*} respectively.  We also define a distance function $| \cdot |$ on $V$, 
$$|x| = |x_1 \omega_1 + \cdots + x_m \omega_m| = \max\limits_{i} | x_i |.$$  This extends to $V^s$, for $s \in \mathbb{N}$: if $\textbf{x} = (x^{(1)}, \dots ,\ x^{(s)}) \in V^s,$ then
$$|\textbf{x}| =  \max\limits_{j} | x^{(j)} |.$$  We note that there will be some constant $c$, dependent only on $k$ and our choice of basis $\omega_1, \dots , \omega_m$, such that \beq\label{cdef} |\pi_i(x)| \leq c|x| \eeq for all $x \in V$ and $1 \leq i \leq m$ (since each $\pi_i$ is linear, this is clear).  Also for any $v,w \in V$, we have 

\beq\label{height} |vw| \ll |v||w|, \quad \txt{Nm}(v) \ll |v|^m \quad \txt{and} \quad |v^{-1}| \ll \frac{|v|^{m-1}}{\txt{Nm}(v)}. \eeq
 
For any point $\textbf{v} \in V^s$, let $\mathfrak{B}(\textbf{v})$ be the box \begin{equation}\label{box2}\mathfrak{B}(\textbf{v}) = \{  \xx \in V^s : |\xx-\textbf{v}|<\rho  \},\end{equation} where $\rho$ is a fixed real number $0 <\rho <1 $.  For a set $\mathcal{{A}} \subset V^s$, and positive real number $P$, we define $P \mathcal{{A}}$ to be the set $\{ \xx \in V^s : P^{-1} \xx \in \mathcal{A} \}$.  

\section{Statement of the Main Lemma}

Consider the smooth quasi-projective variety $Y'$ given by the equation $\eqref{equation1}$ together with the inequalities $\xx \neq 0$, $\yy \neq 0$, $z \neq 0$, $N(\xx) \neq 0$, $N(\yy) \neq 0$.  It is sufficient to prove weak approximation on $Y'$, since weak approximation is a birational invariant on smooth varieties.

We assume equation $\eqref{equation1}$ has a solution in $k_\nu$ for all places $\nu$ of $k$.  Suppose we are given a finite set of places $S$ and a set of local solutions $(\xx_\nu, \yy_\nu , z_\nu) \in Y'(k)$ for each $\nu \in S$.  For any fixed $\eta>0$, our task is to find a $k$-point $(\xx,\yy,z) \in Y'(k)$ such that $$ |x^{(i)} - x_\nu^{(i)}|_\nu < \eta, \quad |y^{(i)} - y_\nu^{(i)}|_\nu<\eta, \quad |z - z_\nu|_\nu<\eta$$ for all $1 \leq i \leq n$, and $\nu \in S$, where $|\cdot|_\nu$ denotes the valuation on $k_\nu$.  Without loss of generality, we can assume that $S$ contains all the infinite places.

For the finite places, we note that by the Chinese Remainder Theorem, finding a rational point which is $\mathfrak{p}$-adically close to some set of $\mathfrak{p}$-adic points, is equivalent to finding an integral point which is restricted to some congruence class modulo some integral ideal.  In our case, we shall let the ideal be $\mathfrak{n}$ as in the notation section.  So we are given $(\xx_{\mathfrak{n}}, \yy_{\mathfrak{n}}, z_{\mathfrak{n}}) \in \mathfrak{o}^{2n+1}$ which is a non-singular solution of $\eqref{equation1}$ modulo $\mathfrak{n}$.  

Our task is now to find a solution $(\xx, \yy, z) \in \mathfrak{o}^{2n+1}$ with \beq\label{cong1} |x^{(i)}- P x^{(i)}_\nu|_\nu< P \eta, \quad |y^{(i)}- Py^{(i)}_\nu|_\nu< P\eta, \quad |z- P z_\nu|_\nu <P \eta  \eeq for each infinite place $\nu$, and \beq\label{cong2} x^{(i)} \equiv x^{(i)}_{\mathfrak{n}} , \quad y^{(i)} \equiv y^{(i)}_{\mathfrak{n}}, \quad z \equiv z_{\mathfrak{n}} \quad \mod{\mathfrak{n}}. \eeq Our main lemma is then the following:

\begin{thma}\label{mainlem} Suppose that for each prime $\mathfrak{p}$ there is a non-singular solution to $\eqref{equation1}$ satisfying $\eqref{cong2}$ in $\mathfrak{p}$-adic integers.  Then $\eqref{equation1}$ has a solution in $\mathfrak{o}^{2n+1}$ satisfying $\eqref{cong1}$ and $\eqref{cong2}$, provided $P$ is sufficiently large. \end{thma}  This will be enough to prove weak approximation on the variety $Y$, and will thus establish Theorem $\ref{mainthm}$ for general $k$.

\section{The Circle Method} We set $$S_1(\alpha)= \sum_{\xx} e(\text{Tr}(\alpha a N(\xx))),$$ $$S_2(\alpha)= \sum_{\yy} e(\text{Tr}(\alpha b N(\yy))),$$ $$S_3(\alpha)= \sum_{z} e(\text{Tr}(\alpha z^n)),$$ with all sums running over modulo classes defined by $\eqref{cong2}$, and inside the dilated boxes $P\mathfrak{B}_1 \subset V^n, P\mathfrak{B}_2 \subset V^n ,P\mathfrak{B}_3 \subset V$ respectively, where $$\mathfrak{B}_1=\mathfrak{B}\left(\bigoplus_{i=1}^{n_1 + n_2} \xx_{\nu_i}\right) \quad \mathfrak{B}_2=\mathfrak{B}\left(\bigoplus_{i=1}^{n_1 + n_2} \yy_{\nu_i}\right) \quad \mathfrak{B}_3=\mathfrak{B}\left(\bigoplus_{i=1}^{n_1 + n_2} z_{\nu_i}\right),$$  $\nu_i$ being the place corresponding to the embedding $\sigma_i$.  

Also, we let $\mathfrak{B}'\subset V^{2n+1}$ be the product $ \mathfrak{B}' = \mathfrak{B}_1 \times  \mathfrak{B}_2 \times  \mathfrak{B}_3$.  From the observation that the constant $c$ in $\eqref{cdef}$ exists, we see that to satisfy $\eqref{cong1}$, it will be sufficient that $( \xx,  \yy,  z) \in P \mathfrak{B}'$, where $\rho = \rho(\eta)$ has been chosen appropriately small in the definition $\eqref{box2}$.  Furthermore, by choosing $\rho$ sufficiently small, we can guarantee that $( \xx,  \yy,  z) \in Y'(K)$

We define $\mathcal{I}$ as: $$\mathcal{I}:= \{ \alpha = \alpha_1 \omega_1 + \cdots + \alpha_m \omega_m \in V : 0 \leq \alpha_i \leq 1 \}. $$  Let $\mathcal{N}( P )$ to be the number of points $(\xx,\yy, z) \in \mathfrak{o}^{2n+1} \cap P \mathfrak{B}'$ which are a solution to $\eqref{equation1}$, and such that the conditions $\eqref{cong2}$ are satisfied.  We have $$\mathcal{N}( P ) = \int_{\mathcal{I}} S_1(\alpha)S_2(\alpha)S_3(-\alpha) d \alpha.$$  
For any $\gamma \in k$, define the denominator ideal of $\gamma$ as $$\mathfrak{a}_\gamma = \{\kappa \in \mathfrak{o} : \kappa \gamma \in \mathfrak{n} \}.$$  
We also set $$\mathfrak{M}_{\gamma}(\theta) = \{ x \in \mathcal{I} : |x-\gamma| \leq P^{-n + m(n-1)\theta} \},$$ for some $\theta>0$ to be fixed later, and define a special subset of $\mathcal{I}$, $$\mathfrak{M} = \mathfrak{M}(\theta)=  \bigcup_{\substack{\gamma \in k \\  \txt{Nm}(\mathfrak{a}_\gamma) \ll P^{m(n-1) \theta}}} \mathfrak{M}_{\gamma},$$ which we shall call the `major arcs'.  We define the `minor arcs' as the compliment of the major arcs, $\mathfrak{m}(\theta)=\mathcal{I} \setminus \mathfrak{M}(\theta)$.

Finally we shall state once and for all that implied constants in any $\ll, \gg$, or $O(\cdot)$ quantifiers, are dependent only only on $k,K, \mathfrak{n}$ with fixed choice of basis, and $\mathfrak{B}$.

\subsection{The Minor Arcs}  First we shall get suitable estimates for $S_1(\alpha)$, and $S_2(\alpha)$.  Note that $N$ is a norm form on $K/\mathbb{Q}$ with $\mathbb{Z}$-basis $\{\omega_i \tau_j\}$.  So the argument of \cite[Lemma 1]{BDL} holds here (in fact we have extra restrictions on our variables but this does not affect the argument).  This results in the estimate

 \beq\label{S1} \int_{\mathcal{I}} |S_j(\alpha)|^2 d \alpha \ll P^{mn+\eps} \eeq  for $j=1,2,$ and any $\eps >0$.  

Now we want to get a bound on $|S_3(\alpha)|$ for $\alpha$ on the minor arcs.

\begin{thma}\label{skinner} Let $\eps > 0$ and suppose $0 < \Delta < 1$.  Either:  \begin{enumerate}[(i)] \item $|S_3(\alpha)| \ll P^{m-\Delta/2^{n-1}+\eps}$ \txt{, or} \item there exists $0 \neq \mu \in \mathfrak{n}, \lambda \in \mathfrak{n}$ such that $$|\mu| \ll P^{(n-1) \Delta} \, \, \txt{and} \, \, |\mu \alpha - \lambda| < P^{-n +(n-1)\Delta}.$$  
\end{enumerate} \end{thma}

\begin{proof}  Consider the sum $$S'_3(\alpha)= \sum_{z} e ( \txt{Tr}(\alpha (z+ z_\mathfrak{n})^n)),$$ where $z$ now runs over the set $ \mathfrak{n} \cap P \mathfrak{B}_3$.  By comparing the domains of summation, we see that $$S_3(\alpha)= S'_3(\alpha)+ O(P^{m-1}),$$ and thus if assumption $(i)$ fails, then it also fails with $S_3(\alpha)$ replaced by $S'_3(\alpha)$.  Put $f(z)= \sum_{i=1}^n \txt{Tr}(\omega_i(z+ z_\mathfrak{n})^n) \omega_i $.  Then $f$ is of the type defined by \cite[Eq 2.6]{Birch}.  Furthermore, in the notation of \cite{Birch}, $$S'_3(\alpha) = \sum_{z \in P \mathfrak{B}_3} e [\alpha \cdot f(z)],$$ so our result is given by  \cite[Lemma 3]{Birch}.  Note that this lemma was for exponential sums over $\mathfrak{o}$ rather than general $\mathfrak{n}$, but it is trivial to generalise to this setting.  \end{proof}

Under the assumption that $\alpha$ satisfies $(ii)$, we have (using $\eqref{height}$)
\begin{align*}  \left| \alpha - \frac{\lambda}{\mu} \right| &\ll |\mu^{-1}||\mu \alpha - \lambda|  \\
&\ll |\mu|^{m-1} P^{-n + (n-1)\Delta}  \\
&\ll P^{-n + m(n-1)\Delta}.  \end{align*}  If we put $\gamma = \frac{\lambda}{\mu}$, we see that $\langle \mu \rangle \subset \mathfrak{a}_\gamma$, and so $$\txt{Nm}(\mathfrak{a}_\gamma) \leq \txt{Nm}\langle \mu \rangle \ll P^{m(n-1)\Delta}.$$  Hence $\alpha \in \mathfrak{M}(\Delta).$  So we deduce

\beq\label{S3} |S_3 (\alpha)|  \ll P^{m-\Delta / 2^{n-1} + \eps}, \eeq for all $\alpha \in \mathfrak{m}(\Delta)$.

Combining this with $\eqref{S1}$ and using Cauchy's inequality we obtain:

\begin{thma}\label{minor}$$\int_{\mathfrak{m}(\Delta)} S_1(\alpha)S_2(\alpha)S_3(-\alpha) d \alpha \ll P^{(n+1)m-\delta}$$ for some $\delta=\delta(\Delta)>0$. \end{thma}

\subsection{The Major Arcs}

For $\kkk = (k^{(1)}, \dots , k^{(2n+1)}) \in \mathfrak{n}^{2n+1}$, we define the function 

\begin{align*} F(k^{(1)}, \dots , k^{(2n+1)}) =  &a N (k^{(1)} + x^{(1)}_\mathfrak{n}, \dots ,  k^{(n)} + x^{(n)}_\mathfrak{n}) \\ & + b N (k^{(n+1)} + y^{(1)}_\mathfrak{n}, \dots , k^{(2n)}  + y^{(n)}_\mathfrak{n}) - (k^{(2n+1)}+z_\mathfrak{n})^n. \end{align*}  Note that the assumption of Lemma $\ref{mainlem}$ is equivalent to the assumption that $F(\kkk)=0$ has a non-singular solution in $\mathfrak{n}_\mathfrak{p}$ for every prime $\mathfrak{p}$.

Put $$S_\gamma =\txt{Nm}(\mathfrak{a}_\gamma)^{-(2n+1)} \sum_{\kkk \mod {\mathfrak{n} \mathfrak{a}_\gamma}} e( \textrm{Tr}(\gamma F(\kkk))),$$ the sum being over $\kkk \in \mathfrak{n}^{2n+1}$.  We then define $$\mathfrak{S}(\Delta) = \sideset{}{'}\sum_{\txt{Nm}(\mathfrak{a}_\gamma)\leq P^{\Delta}}S_\gamma, $$ where the dash indicates that only one $\gamma$ should be taken from each equivalence class modulo $\mathfrak{n}$.  We call this the \textit{singular series}.  Finally, put $$\mathfrak{I}(\Delta) = \int_{|\beta|<P^{\Delta}} \int_{\mathfrak{B}'} e(\textrm{Tr}(\beta F(\kkk))) d\kkk d\beta.$$  This is the \textit{singular integral}.

\begin{thma}\label{major} For $\Delta$ sufficiently small,
  $$ \int_{\mathfrak{M}(\Delta)}S_1(\alpha)S_2(\alpha)S_3(-\alpha) d \alpha = \mathfrak{S}(\Delta)  \mathfrak{I}(\Delta) P^{(n+1)m}+O(P^{(n+1)m - \delta}),  $$ for some $\delta = \delta(\Delta) >0$. \end{thma}
  
  \begin{proof}This follows from \cite[Lemma 7]{SK}.  \end{proof}
  
Combining this lemma with Lemma $\ref{minor}$, we see \begin{align*} \mathcal{N}(P) &= \int_{\mathfrak{M}(\Delta)} S(\alpha) d \alpha + \int_{\mathfrak{m}(\Delta)} S(\alpha) d \alpha \\ 
&= \mathfrak{S}(\Delta) \mathfrak{I}(\Delta) P^{(n+1)m}+O(P^{(n+1)m - \delta}). \end{align*}   So all that remains to show is that under the assumption of Lemma $\ref{mainlem}$, $\mathfrak{S}(\Delta)$ and $\mathfrak{I}(\Delta)$ have strictly positive limits as $P \rightarrow \infty$.

\begin{thma}\label{integral} For our box $\mathfrak{B}'$ chosen as before, $\mathfrak{I}(\Delta) \rightarrow \mathfrak{I}_0$, a constant as $P \rightarrow \infty$.  Furthermore $\mathfrak{I}_0>0$. \end{thma}

\begin{proof}  We define the polynomial $$F^{*}(\xx) = F(x^{(1)}_1\omega_1 + \cdots + x^{(1)}_m\omega_m,   \, \, \dots,   \, \, x^{(s)}_1\omega_1 + \cdots + x^{(s)}_m\omega_m),$$ considered as a real polynomial in the $sm$ variables $ \{ x^{(1)}_{1}, \dots , x^{(s)}_{m} \}$.  In the definition of $\mathfrak{I}$, we can just as easily think of the inner integral being over $\mathbb{R}^{mn}$ with $F$ replaced by $F^*$, and the outer integral as being over the real variables $\beta_1 , \dots , \beta_m$, where $\beta = \beta_1 \omega_1 + \cdots + \beta_m \omega_m$.  Then this lemma is routine, and indeed an argument analagous to the one used in \cite{HBS} can be used.   The key point is that the box is centred at a nonsingular point in $V^n$ (note that a non-singular solution to $F$ in $V^{n}$ corresponds to a non-singular solution to $F^*$ in $\mathbb{R}^{mn}$).

 \end{proof}

\begin{thma}\label{series} We have \begin{enumerate}[(i)] \item $\mathfrak{S}(\infty)$ exists, \item $\mathfrak{S}(\Delta)-\mathfrak{S}(\infty) \ll P^{-\zeta}$, for some positive $\zeta= \zeta(\Delta)$, and \item  $\mathfrak{S}(\infty) > 0$. \end{enumerate} \end{thma}  We follow the arguments of \cite{HBS}.  Consider the sum $$T_1(\gamma)= \sum_{\kkk_1 \mod { \mathfrak{a}_\gamma}} e(\txt{Tr}(\gamma F_1(\kkk_1))),$$ where $F_1(\kkk_1)= a N (k_1 + x^{(1)}_\mathfrak{n}, \dots ,  k_n + x^{(n)}_\mathfrak{n})$.  Define $T_2$ analogously, and set $$T_3(\gamma) =  \sum_{k \mod { \mathfrak{a}_\gamma}} e(\txt{Tr}(\gamma (k + z_{\mathfrak{n}})^n)).$$  
Then clearly $S_\gamma =\txt{Nm}(\mathfrak{a}_\gamma)^{-(2n+1)} T_1(\gamma)T_2(\gamma)T_3(-\gamma)$.  We will consider the dyadic range:  $$\mathfrak{S}_R=\sideset{}{'}\sum_{R/2 < \txt{Nm}(\mathfrak{a}_\gamma) \leq R} \txt{Nm}(\mathfrak{a}_\gamma)^{-(2n+1)} T_1(\gamma) T_2(\gamma) T_3(- \gamma).$$ 

If we repeat the argument of Lemmas $\ref{major}$ and $\ref{integral}$ with $|S_1(\alpha)|^2$ in place of $S_1(\alpha)S_2(\alpha)S_3(-\alpha)$, we find that

$$\sideset{}{'}\sum_{\txt{Nm}(\mathfrak{a}_\gamma)\ll P^\Delta} \int_{\mathfrak{M}_\gamma(P^\Delta)} |S_1(\alpha)|^2d\alpha = \Sigma_1 \mathfrak{I}_1+O(P^{mn-\delta}),$$ for some $\delta =\delta(\Delta) >0$, and where 

$$\Sigma_1 = \sideset{}{'}\sum_{\txt{Nm}(\mathfrak{a}_\gamma)\leq P^{\Delta}}\txt{Nm}(\mathfrak{a}_\gamma)^{-2n} |T_1(\gamma)|^2,$$ and

$$\mathfrak{I}_1 \sim CP^{mn}$$ for some positive constant $C$.  But 

$$\sideset{}{'}\sum_{\txt{Nm}(\mathfrak{a}_\gamma)\ll P^\Delta} \int_{\mathfrak{M}_\gamma(P^\Delta)} |S_1(\alpha)|^2d\alpha \leq \int_{\mathcal{I}} |S_1(\alpha)|^2 d\alpha \ll P^{mn+\eps} $$ by $\eqref{S1}$.  Note that the estimate holds for any $P \geq 1$, and $\eps>0$.  So if we choose $P$ such that $P^{\Delta} = R$, and put $\varpi = \eps / \Delta$, we see that

$$\sideset{}{'}\sum_{\txt{Nm}(\mathfrak{a}_\gamma)\leq R}\txt{Nm}(\mathfrak{a}_\gamma)^{-2n} |T_1(\gamma)|^2 \ll R^\varpi,$$ for any $R \geq 1$, and $\varpi > 0$.  Similarly we have 

$$\sideset{}{'}\sum_{\txt{Nm}(\mathfrak{a}_\gamma)\leq R}\txt{Nm}(\mathfrak{a}_\gamma)^{-2n} |T_2(\gamma)|^2 \ll R^{\varpi},$$ and so \beq\label{S1S2}\sideset{}{'}\sum_{R/2 < \txt{Nm}(\mathfrak{a}_\gamma) \leq R} \txt{Nm}(\mathfrak{a}_\gamma)^{-2n} |T_1(\gamma) T_2(\gamma)| \ll R^{\varpi} \eeq by Cauchy's inequality.

Now we bound $T_3(\gamma)$.  Let $N=\txt{Nm}(\mathfrak{a}_\gamma)$, and note that \begin{align*}|N^{-1}T_3(\gamma)| &= |N^{-m} \sum_{z  \mod  \langle N \rangle} e(\txt{Tr}(\gamma(z + z_{\mathfrak{n}})^n) | \\ 
&=  |N^{-m} \sum_{z \in \mathfrak{n} \cap N \mathfrak{B}''} e(\txt{Tr}(\gamma(z + z_{\mathfrak{n}})^n) |, \end{align*}
where $\mathfrak{B}'' = \{ x \in V : 0 \leq x_i < 1 \, \, \forall \, i \}$.  So now we can use Lemma $\ref{skinner}$, replacing $S_3$ with the exponential sum on the last line, taking $P=N$, and $\Delta < 1 / m (n-1)$.  
If alternative $(i)$, holds we have \begin{align}\label{S3est}|N^{-1}T_3(\gamma)| &\ll N^{-m}N^{m- \Delta/ 2^{n-1}+ \delta} = N^{- \Delta/ 2^{n-1}+ \delta} \end{align}  for any $\delta>0$.  On the other hand, alternative $(ii)$ implies the existence of some $\mu, \lambda \in \mathfrak{n}$ where $$0 < |\mu| \ll N^{(n-1)\Delta},$$ and $$|\mu \gamma - \lambda|< N^{-n + (n-1)\Delta}.$$  
Note that $\mathfrak{a}_\gamma (\mu \gamma - \lambda) \subseteq \mathfrak{n},$ so that if $$\mu \gamma - \lambda = \sum \theta_i \omega_i,$$ then $\theta_i \txt{N} \in \mathbb{Z}$ for all $i$.  But $|\theta_i N| < N^{(n-1)(\Delta-1)} < 1$, and so $\theta_i = 0$ for all $i$.  
It follows that $\mu \in \mathfrak{a}_\gamma $, so that $ N | \txt{Nm}(\mu)$.  But $$\txt{Nm}(\mu) \ll |\mu|^m \ll N^{m(n-1)\Delta} \ll N^{1-\eps},$$ for some positive $\eps$, and since $\txt{Nm}(\mu) \neq 0$, this is a contradiction if $\txt{N}$ is large enough.  Therefore $\eqref{S3est}$ holds, and combining this with $\eqref{S1S2}$, we arrive at the estimate $$\mathfrak{S}_R\ll R^{\varpi- \Delta/ 2^{n-1}+ \delta}.$$
Since $\varpi$, $\delta$ were arbitrary, $(i)$ and $(ii)$ of our lemma follow immediately.  
The proof of $(iii)$ is routine. For any prime ideal $\mathfrak{p}$, we define $$\mu(\mathfrak{p})= \sum_{j=1}^{\infty} \sum_{\mathfrak{a}_\gamma = \mathfrak{p}^j} S_\gamma,$$ and then $$\mathfrak{S}(\infty) = \prod_\mathfrak{p} \mu(\mathfrak{p}).$$  
Standard arguments show that the assumption that $F$ has a non-singular solution in each $\mathfrak{n}_\mathfrak{p}$ implies that each $\mu(\mathfrak{p})>0$, and that the product is strictly positive.  This completes the proof.
\begin{small}

\end{small}


\begin{thebibliography}{99}
\bibitem{Birch}
{\sc Birch, B. J.}
\textit{Waring's problem in algebraic number fields.} Proc. Cambridge Philos. Soc. \textbf{57} (1961) 449--459. 
\bibitem{BDL}
{\sc Birch, B. J.; Davenport, H.; Lewis, D. J.}
\textit{The addition of norm forms.} Mathematika \textbf{9} (1962) 75--82.
\bibitem{CHS}
{\sc Colliot-Th\'{e}l\`{e}ne, J.-L. ; Harari, D.; Skorobogatov, A.}
\textit{Valeurs d'un polynôme à une variable représentées par une norme.} Number theory and algebraic geometry, 69--89,
London Math. Soc. Lecture Note Ser. 303, Cambridge Univ. Press, Cambridge, (2003).
\bibitem{CTSal}
{\sc Colliot-Th\'{e}l\`{e}ne, J.-L. ; Salberger, P.}
\textit{Arithmetic on some singular cubic hypersurfaces.} Proc. London Math. Soc. (3) \textbf{58} (1989), no. 3, 519–-549. 
\bibitem{HBS}
{\sc Heath-Brown, D. R.; Skorobogatov, A.}
\textit{Rational solutions of certain equations involving norms.} Acta Math. \textbf{189}, no. 2 (2002), 161--177.
\bibitem{SK}
{\sc  C. M Skinner}
\textit{Forms over number fields and weak approximation.} Compisitio Mathematica \textbf{106}
(1997), 11--29.

\end{thebibliography}
\end{document}